\documentclass[12pt]{amsart}
\usepackage{amsmath,amssymb,amsthm,amsrefs}
\usepackage{enumitem}
\usepackage{color,hyperref}
\hypersetup{colorlinks,breaklinks,
	linkcolor=blue,urlcolor=blue,
	anchorcolor=blue,citecolor=blue}
\usepackage{cleveref}
\theoremstyle{plain}
\newtheorem{theorem}{Theorem}[section]
\newtheorem{proposition}[theorem]{Proposition}

\theoremstyle{definition}

\newtheorem{remark}[theorem]{Remark}

\makeatletter
\@addtoreset{equation}{section}
\makeatother

\makeatletter
\newcommand{\Spvek}[2][r]{%
  \gdef\@VORNE{1}
  \left(\hskip-\arraycolsep%
    \begin{array}{#1}\vekSp@lten{#2}\end{array}%
  \hskip-\arraycolsep\right)}

\def\vekSp@lten#1{\xvekSp@lten#1;vekL@stLine;}
\def\vekL@stLine{vekL@stLine}
\def\xvekSp@lten#1;{\def\temp{#1}%
  \ifx\temp\vekL@stLine
  \else
    \ifnum\@VORNE=1\gdef\@VORNE{0}
    \else\@arraycr\fi%
    #1%
    \expandafter\xvekSp@lten
  \fi}
\makeatother

\begin{document}
\title[Capacity and measure approximations for Schr\"{o}dinger operators]{Capacity and measure approximations for Schr\"{o}dinger operators\\}

\author{Burak Hat\.{i}no\u{g}lu}
\address{Department of Mathematics, Michigan State University, East Lansing MI 48824, U.S.A.}
\email{hatinogl@msu.edu}

\author{Svetlana Jitomirskaya}
\address{Department of Mathematics, University of California, Berkeley CA 94720, U.S.A.}
\email{sjitomi@berkeley.edu}




\begin{abstract}
We prove that logarithmic capacity convergence for phase-union spectra of quasi-periodic Schr\"{o}dinger operators in the zero Lyapunov exponent regime is robust, requiring only continuity of the potential. Let $S^+(p/q)$ denote the union, over the phase, of the spectra at rational frequency $p/q$. We show that if the Lyapunov exponent vanishes on the spectrum $\Sigma(\alpha)$ at an irrational frequency $\alpha$, then for every sequence $p_n/q_n\to\alpha$, the logarithmic capacities 
$Cap(S^+(p_n/q_n))\longrightarrow Cap(\Sigma(\alpha)).$
We also prove convergence of the corresponding harmonic measures.

As a consequence, the equilibrium measures of $S^+(p_n/q_n)$ converge in the weak$^*$ topology to the density of states measure of the quasi-periodic Schr\"odinger operator. We extend these results to multi-frequency Schr\"odinger operators and prove analogous convergence theorems, for logarithmic capacity, harmonic measure, and equilibrium measure, for ergodic Schr\"odinger operators in a general setting where the almost sure spectrum is approximated in the Hausdorff metric by union spectra of periodic operators. This abstract formulation applies, in particular, to uniformly almost periodic potentials along sequences of almost periods. We also provide counterexamples when the limiting frequency is rational.

\end{abstract}

\maketitle

\section{Introduction}\label{Intro}

If $(\Omega,d\mu)$ is a probability measure space and $T : \Omega \rightarrow \Omega$ is an ergodic transformation, let
\begin{equation}
    (H_{\omega}\psi)_n = \psi_{n+1} + \psi_{n-1} + g(T^n\omega)\psi_n,
\end{equation}
where $g$ is a bounded, measurable function from $\Omega$ to $\mathbb{R}$. Then $\{H_{\omega}\}_{\omega \in \Omega}$ defines an ergodic family of Schr\"{o}dinger operators. 

For an ergodic family $\{H_{\omega}\}_{\omega \in \Omega}$, there exists a set $\Sigma$ independent of $\omega$ such that $\Sigma$ is the spectrum of $H_{\omega}$ for almost every $\omega$ with respect to the measure $\mu$ \cite{CFKS87}. The set $\Sigma$ is the almost sure spectrum of $H_{\omega}$. The same non-randomness is valid for the absolutely continuous, singular continuous, and pure point spectra of $H_{\omega}$ \cite{CFKS87}.

Letting $\Omega$ be $\mathbb{R}/\mathbb{Z}$ and letting the ergodic transformation $T$ be addition with $\alpha \in \mathbb{R}/\mathbb{Z}$, we get quasi-periodic Schr\"{o}dinger operators.

For an irrational $\alpha \in \mathbb{T}:= \mathbb{R}/\mathbb{Z}$, the quasi-periodic Schr\"odinger operator $$H_{\alpha,\theta}:~ l^2(\mathbb{Z}) \rightarrow l^2(\mathbb{Z})$$ with continuous potential $v \in \mathcal{C}(\mathbb{T},\mathbb{R})$ is given by
\begin{equation}\label{1}
    \big(H_{\alpha,\theta}\psi\big)_n = \psi_{n-1} + \psi_{n+1} + v(\alpha n + \theta)\psi_n,
\end{equation}
where $\theta \in \mathbb{T}$ and $\alpha \in \mathbb{T}$ are called the phase and frequency, respectively. The spectrum and the absolutely continuous spectrum of $H_{\alpha,\theta}$ are denoted by $\sigma(\alpha,\theta)$ and $\sigma_{ac}(\alpha,\theta)$, respectively, and are independent of $\theta$ for any irrational frequency $\alpha$. Therefore, we use the notation 
\begin{equation}\label{unionSpectra}
    \Sigma(\alpha) := \sigma(\alpha,\theta), \qquad \theta \in \mathbb{T}
\end{equation}
and
\begin{equation}\label{intersectionSpectra}
    \Sigma_{ac}(\alpha) := \sigma_{ac}(\alpha,\theta), \qquad \theta \in \mathbb{T}
\end{equation}
for irrational $\alpha$.

When the frequency is rational, $H_{p/q,\theta}$ becomes a periodic operator for any phase, with a purely absolutely continuous spectrum of a band-gap structure. In this case, $\sigma(p/q,\theta)$ depends on $\theta$, so we consider the union spectrum
\begin{equation}
    S_+(p/q) = \bigcup_{\theta \in \mathbb{T}}\sigma(p/q,\theta)
\end{equation}
and the intersection spectrum
\begin{equation}
    S_-(p/q) = \bigcap_{\theta \in \mathbb{T}}\sigma_{ac}(p/q,\theta)
\end{equation}
following \cite{AMS90} and \cite{JM12}. 

Dynamical aspects of $H_{\alpha,\theta}$ follows through the transfer matrix cocycles. Transfer matrices $A^E$ describe how solutions of the eigenvalue difference equation 
\begin{equation}
    H_{\alpha,\theta}\psi = E \psi
\end{equation}
are generated iteratively. Schr\"{o}dinger cocycles $(\alpha,A^E)$ provide the dynamical framework to capture this observation. On the other hand, the density of states measure $dk_{\alpha}$ carries out the spectral properties of $H_{\alpha,\theta}$. Its cumulative distribution $k_{\alpha}$ is called the integrated density of states. An important connection between the dynamical and spectral sides of $H_{\alpha,\theta}$ is 
\begin{equation}
    k_{\alpha}(E) = 1 - 2\rho(\alpha , E),
\end{equation}
where $\rho$ is the rotation number of the Schr\"{o}dinger cocycle $(\alpha,A^E)$. We introduce these spectral and dynamical concepts in the second part of Section \ref{Prelim}. 

Another dynamical description of $H_{\alpha,\theta}$ is through its spectrum. The spectrum $\Sigma(\alpha)$ is the bifurcation set for the energy-indexed family of cocycles $(\alpha,A^E)$. In other words, the spectrum is the set of $E$, where the Schr\"{o}dinger cocyle fails to be uniformly hyperbolic. This result is known as Johnson's theorem \cite{J86,JM17}.

Quasi-periodic Schr\"odinger operators appear in the modeling of the influence of an external magnetic field on a crystal layer. Their study gained prominence with work on the almost Mathieu operator, which is a special case with potential $v(x) = 2\lambda \cos(2\pi x)$, see \cite{flu}. The almost Mathieu operator is classified as subcritical, critical, and supercritical depending on the coupling constant $\lambda,$ with $\lambda=1$ being critical. This classification of energies was generalized to quasi-periodic Schr\"odinger operators with analytic potentials by Avila \cite{Avi15}, see Section \ref{Prelim} for details. 

In experiments and numerics, the irrational $\alpha$ are usually replaced by their rational approximations, so the corresponding convergence questions become important. When $p_n/q_n\to\alpha$ and $\alpha$ is irrational, then the sets $S_+(p_n/q_n)$ and $S_-(p_n/q_n)$ do converge to $\Sigma(\alpha)$ and $\Sigma_{ac}(\alpha)$ in Hausdorff metric \cite{AS83} and also in measure \cite{JM12}, see Section \ref{Literature} for details. However, this alone does not, in general, imply the convergence of logarithmic capacities \cite{KK22}, harmonic measures or equilibrium measures \cite{Tot24}.

This paper is the second of a series of papers in which we study convergence of logarithmic capacities and harmonic measures for union spectra of periodic  approximants. In the first paper \cite{HJ24}, we proved that for analytic potentials, convergence of the capacities holds when the $\Sigma(\alpha)$ has no supercritical energy (or equivalently Lyapunov exponent is zero on $\Sigma(\alpha)$), using Avila's global theory \cite{Avi15} and Chebyshev polynomials. Here we develop a robust way to obtain this result, using a purely potential-theoretic argument, allowing to replace analyticity with merely continuity, and extending it to multifrequency, and general ergodic potentials under the same zero Lyapunov exponent assumption.  This also yields convergence of harmonic measures and equilibrium measures; in the zero Lyapunov exponent regime the limiting equilibrium measure is the density of states measure.

From a dynamical systems point of view, rational frequency approximation is more than a technical limiting procedure. The rational operators provide periodic models which shadow the irrational cocycle dynamics, and the corresponding union spectra $S^+(p/q)$ may be viewed as periodic approximations to the irrational bifurcation set $\Sigma(\alpha)$. Our results show that this shadowing persists at the level of logarithmic capacity. In the zero Lyapunov exponent regime, the periodic equilibrium measures also converge to the density of states measure, so the potential-theoretic equilibrium of the periodic band sets converges to the natural spectral measure of the limiting ergodic system.

Periodic approximation has also appeared in a broader ergodic context. In particular, Last studied periodic Jacobi matrices obtained by cutting finite pieces of an ergodic potential and repeating them, and related the limsup of their spectra to the absolutely continuous spectrum. Our use of periodic approximants is different in emphasis: for the potential-theoretic conclusions below, the key topological input is Hausdorff convergence of the approximating union spectra to the limiting almost sure spectrum.

Although the ergodic formulation is the most general one, we present the one-frequency quasi-periodic results first. This is the setting in which the motivating approximation problem naturally arises: irrational frequencies are replaced by rational ones, and the relevant approximating objects are the phase-union spectra $S^+(p/q)$, also where examples with absolutely continuous spectrum are abundant. Treating this case separately also keeps the connection with Lyapunov exponents, rotation numbers, the density of states, and earlier results on rational frequency approximation explicit. The multi-frequency and ergodic results are then obtained by isolating the potential-theoretic mechanism behind the proof and applying it whenever the required Hausdorff convergence of the approximating spectra is available.

More precisely, we prove capacity convergence for one-frequency quasi-periodic Schr\"odinger operators with continuous potentials when the Lyapunov exponent vanishes on the spectrum (Section~4). This implies convergence of harmonic measures and equilibrium measures (Section~5). We then give multi-frequency analogues of these results (Section~6). Finally, we formulate an ergodic version in which the quasi-periodic structure is replaced by the assumption that the union spectra of a family of periodic approximants converge to the almost sure spectrum in the Hausdorff metric (Section~7). We also show that this Hausdorff hypothesis is satisfied for uniformly almost periodic potentials along sequences of almost periods. Section~8 gives counterexamples when the limiting frequency is rational.

Capacity is a set function that arises in potential theory as an analogue of the physical concept of the electrostatic capacity. The capacity of a set in the complex plane is called logarithmic capacity, since in two-dimensions the Newtonian potential kernel becomes logarithmic. 

The logarithmic capacity of a set in the complex plane is obtained by considering the infimum of the energy over the Borel probability measures supported on this set. The minimizing measure is called the equilibrium measure of the set. The extremal nature of the capacity is observed by two characterizations: one reflects the geometry of the set (diameter), and the other depends on the extremal polynomials on the set (Chebyshev polynomials). 

The nth diameter of a compact set $K$ in the complex plane is defined as the maximum of the geometric mean of the distances between all pairs of $n$ points, where the maximum is taken over the subsets of $K$ with cardinality $n$. The limit of the nth diameter of $K$ as $n$ tends to $\infty$ is called the transfinite diameter of $K$, which is the same as its logarithmic capacity.

The nth Chebyshev polynomial of a compact set $K$ in the complex plane is defined as the unique nth degree monic polynomial minimizing the supremum norm on $K$ over monic polynomials of the same degree. That minimizing norm value is the nth Chebyshev number of $K$. The limit of the nth root of the nth Chebyshev number of $K$ is again nothing but the logarithmic capacity of $K$.

Harmonic measure is a family of measures that extends the Dirichlet problem to proper subdomains (non-empty, connected, open subsets) of the extended complex plane $\overline{\mathbb{C}} := \mathbb{C} \cup \{\infty\}$. By definition, it is connected with harmonic functions and potential theory. It also appears in the study of Brownian motion and conformal maps \cite{GM05}. The harmonic measure at infinity is a special measure as it becomes the equilibrium measure of the boundary. If the boundary is a non-polar (nonzero capacity) set, then the harmonic measure is unique.

We now state two broad versions of our main conclusions. The one-frequency quasi-periodic theorems proved in Sections~4 and~5 are the motivating case; the following statements record the corresponding multi-frequency and abstract ergodic formulations.

\begin{theorem}
    Let $\vec{\alpha} \in \mathbb{T}^N$ such that $\{1,\alpha_1,\ldots,\alpha_N\}$ is rationally independent, and
    $$\vec{\alpha}_n := \bigg(\frac{p_{1,n}}{q_{1,n}},\frac{p_{2,n}}{q_{2,n}},\dots,\frac{p_{N,n}}{q_{N,n}}\bigg) \longrightarrow \vec{\alpha}$$ as $n \rightarrow \infty$, where $(p_{i,n},q_{i,n}) = 1$ for $1 \leq i \leq N$, $n \in \mathbb{N}$. Also let $H_{\vec{\alpha},\vec{\theta}}$ be a multi-frequency quasi-periodic Schr\"{o}dinger operator with a continuous potential $v$ and zero Lyapunov exponent on its spectrum $\Sigma(\vec{\alpha})$ such that
    \begin{itemize}
    \item $\Sigma(\vec{\alpha})$ denotes the spectrum of $H_{\vec{\alpha},\vec{\theta}}$,
    \item $S_+(\vec{\alpha}) := \bigcup_{\vec{\theta} \in \mathbb{T}^N}\sigma(\vec{\alpha},\vec{\theta})$, where $\sigma(\vec{\alpha},\vec{\theta})$ is the spectrum of $H_{\vec{\alpha},\vec{\theta}}$.
\end{itemize}
Then we have the following convergence results.
\begin{enumerate}
    \item \textbf{spectral convergence in logarithmic capacity:}
    \begin{equation}
        \lim_{n \rightarrow \infty} Cap\big(S_+(\vec{\alpha}_n)\big) = Cap\big(\Sigma(\vec{\alpha})\big).
    \end{equation}
    
    \item \textbf{spectral convergence in harmonic measure:} if $\omega_D$ denotes the harmonic measure for the domain $D \subset \overline{\mathbb{C}}$, then for any $z \notin \Sigma(\vec{\alpha})$
    \begin{equation}
        \omega_{\overline{\mathbb{C}} \setminus S_+(\vec{\alpha}_n)}(z,\cdot) \rightarrow \omega_{\overline{\mathbb{C}} \setminus \Sigma(\vec{\alpha})}(z,\cdot)
     \end{equation}
    in the weak*-topology as $n \rightarrow \infty$.
     
    \item \textbf{approximations of the density of states measure:} 
    \begin{equation}
        \mu_{S_+(\vec{\alpha}_n)} \rightarrow dk_{\vec{\alpha}}
     \end{equation}
    in the weak*-topology as $n \rightarrow \infty$, where $\mu_{S_+(\vec{\alpha}_n)}$ is the equilibrium measure of $S_+(\vec{\alpha}_n)$ and $dk_{\vec{\alpha}}$ is the density of states measure of $H_{\vec{\alpha},\vec{\theta}}$.
\end{enumerate}
    
\end{theorem}

\begin{theorem}
Let $\{H_{\omega}\}_{\omega \in \Omega}$ be a family of ergodic Schr\"{o}dinger operators. Also let $\{\{H_{\omega,n}\}_{\omega \in W}\}_{n}$ be a sequence of families of periodic Schr\"{o}dinger operators for $W \subseteq \Omega$ such that
\begin{itemize}
    \item $\Sigma$ denotes the almost sure spectrum of $H_{\omega}$,
    \item $S_+(n) := \bigcup_{\omega \in W} \sigma(H_{\omega,n}) $, where $\sigma(H_{\omega,n})$ is the spectrum of $H_{\omega,n}$,
    \item The Lyapunov exponent of $H_{\omega}$ is zero on $\Sigma$, and
    \item $S_+(n)$ converges to $\Sigma$ in Hausdorff metric as $n \rightarrow \infty$.
\end{itemize}
Then we have the following convergence results.
\begin{enumerate}
    \item \textbf{spectral convergence in logarithmic capacity:}
    \begin{equation}
        \lim_{n \rightarrow \infty} Cap\big(S_+(n)\big) = Cap\big(\Sigma\big).
    \end{equation}
    
    \item \textbf{spectral convergence in harmonic measure:} if $\omega_D$ denotes the harmonic measure for the domain $D \subset \overline{\mathbb{C}}$, then for $z \notin \Sigma$
    \begin{equation}
        \omega_{\overline{\mathbb{C}} \setminus S_+(n)}(z,\cdot) \rightarrow \omega_{\overline{\mathbb{C}} \setminus \Sigma}(z,\cdot)
     \end{equation}
     in the weak*-topology as $n \rightarrow \infty$.
     
    \item \textbf{approximations of the density of states measure:} 
    \begin{equation}
        \mu_{S_+(n)} \rightarrow dk
     \end{equation}
    in the weak*-topology as $n \rightarrow \infty$, where $\mu_{S_+(n)}$ is the equilibrium measure of $S_+(n)$ and $dk$ is the density of states measure of $H_{\omega}$.
\end{enumerate}
\end{theorem}

In Section~7 we also verify the Hausdorff convergence hypothesis in a standard deterministic situation. If $V$ is uniformly almost periodic and $q_j$ is a sequence of almost periods, then the spectra of the $q_j$-periodic potentials obtained by repeating length-$q_j$ blocks converge, after taking the union over the hull, to the common hull spectrum in the Hausdorff metric. Therefore, under the additional assumption that the Lyapunov exponent vanishes on this spectrum, the capacity, harmonic measure, and equilibrium measure conclusions of Theorem~1.2 apply to uniformly almost periodic potentials along such sequences of almost periods.

The paper is organized as follows. Section~2 reviews the necessary background from logarithmic potential theory and the spectral theory of discrete Schr\"odinger operators. Section~3 summarizes the convergence results for rational approximants that we use, including convergence of the union and intersection spectra in the Hausdorff metric and in Lebesgue measure. Section~4 proves the convergence of logarithmic capacities in the one-frequency quasi-periodic setting. Section~5 derives the corresponding convergence of harmonic and equilibrium measures. Section~6 extends the results to multi-frequency quasi-periodic Schr\"odinger operators. Section~7 gives the abstract ergodic formulation and proves an almost-periodic corollary by verifying the Hausdorff convergence hypothesis along almost periods. Section~8 discusses counterexamples when the limiting frequency is rational.

\section{Preliminaries}\label{Prelim}

\subsection{Logarithmic potential theory}

Our results discuss convergence of logarithmic capacities and harmonic measures for rational frequency approximates, so let us start by reviewing some basics of logarithmic potential theory to recall these main concepts, which are found, e.g. in \cite{GM05} or \cite{Ran95}.

\subsubsection{Logarithmic capacity}

If $\mu$ is a finite Borel measure, supported on a compact subset of the complex plane, then the \textit{logarithmic potential} of $\mu$ is the function $$U^{\mu}:\mathbb{C}\rightarrow(-\infty,\infty]$$ 
defined by 
\begin{equation}
U^{\mu}(z):=\int\log\frac{1}{|z-w|}d\mu(w)
\end{equation}
and the \textit{logarithmic energy} of $\mu$ is $I(\mu)\in(-\infty,\infty]$, defined by
\begin{equation}
I(\mu):=\int U^{\mu}(z)d\mu(z)=\int\int\log\frac{1}{|z-w|}d\mu(w)d\mu(z).
\end{equation}
Letting $K$ be a compact subset of $\mathbb{C}$ and $M(K)$ be the set of Borel probability measures compactly supported within $K$, we define the \textit{equilibrium measure} for $K$, $\mu_{K}\in M(K)$ as
\begin{equation}
    I(\mu_{K})=\displaystyle\inf_{\mu\in M(K)}I(\mu).
\end{equation}
The \textit{logarithmic capacity} of a subset $E$ of the complex plane is given by
\begin{equation}
Cap(E):=\sup_{\mu\in M(E)}\exp\big(-I(\mu)\big).
\end{equation}
In particular, if $K$ is compact with equilibrium measure $\mu_{K}$, then 
\begin{equation}
    Cap(K)=\exp\big(-I(\mu_{K})\big).
\end{equation}
A Borel set $E\in\mathbb{C}$ is $polar$ if $Cap(E) = 0$, i.e. $I(\mu)=+\infty$ for every $\mu\in M(E)$ with supp$\mu\subseteq E$. A property is said to hold \textit{nearly (or quasi) everywhere} if it holds up to a polar set. 
Two important examples are $Cap([-1,1])=1/2$ and $Cap(\overline{\mathbb{D}})=1$.

Some elementary properties of the logarithmic capacity (Theorem 5.1.2 and Theorem 5.1.3 in \cite{Ran95}) are as follows:
\begin{enumerate}
\item If $E_{1}\subseteq E_{2}$, then $Cap(E_{1})\leq Cap(E_{2})$.
\item If $E\subset\mathbb{C}$, then $Cap(E)=\sup\{Cap(K)$ $|$ $compact$ $K\subset E\}$.
\item If $E\subset\mathbb{C}$, then $Cap(a E+b)=|\alpha|Cap(E)$ for $a,b\in\mathbb{C}$.
\item If $K$ is a compact subset of $\mathbb{C}$, then $Cap(K) = Cap(\partial_e K)$, where $\partial_e$ denotes the exterior boundary.
\item If $K_1 \supset K_2 \supset K_3 \supset \cdots$ are compact subsets of $\mathbb{C}$ and $K = \cap_n K_n$, then
    \begin{equation}\label{CapCompactConvergence}
    Cap(K) = \lim_{n \rightarrow \infty} Cap(K_n).
    \end{equation}
\item If $B_1 \subset B_2 \subset B_3 \subset \cdots$ are Borel subsets of $\mathbb{C}$ and $B = \cup_n B_n$, then
    \begin{equation}
    Cap(B) = \lim_{n \rightarrow \infty} Cap(B_n).
    \end{equation}
\end{enumerate}

Two characterizations of the logarithmic capacity are important in understanding its extremal nature. 

The first characterization is given by the \textit{transfinite diameter}. The \textit{nth diameter} of a compact set $K$, denoted by $\delta_{n}(K)$, is the supremum of the geometric mean of distances between pairs among n-tuples of $K$, i.e.
\begin{equation}
\delta_{n}(K):=\sup_{z_{1},\cdots,z_{n}\in K}\Big\{\prod_{j\textless k}|z_{j}-z_{k}|^{\frac{2}{n(n-1)}}\Big\}.
\end{equation}
An n-tuple $z_{1},\cdots,z_{n} \in K$ for which the supremum is attained is called a \textit{Fekete n-tuple} for $K$. A classical theorem going back to Fekete and Szeg\"{o} (Theorem 5.5.2 in \cite{Ran95}) says that the sequence $\{\delta_{n}(K)\}_{n\geq2}$ is decreasing and
\begin{equation}
\lim_{n\rightarrow\infty}\delta_{n}(K)=Cap(K).
\end{equation}

The second characterization of logarithmic capacity is given by the \textit{Chebyshev polynomials}. Let $K$ be a compact subset of $\mathbb{C}$ and $T_{n,K}$ be the monic polynomial of degree n such that 
\begin{equation}
    ||{T}_{n,K}||_{K}\leq||P||_{K}
\end{equation}
for any monic polynomial $P$ of degree $n$, where $||\cdot||_{K}$ denotes the uniform (supremum) norm over $K$. Then $T_{n,K}$ is called the \textit{nth Chebyshev polynomial on $K$} and $||{T}_{n,K}||_{K}$ is called the \textit{nth Chebyshev number of K}, denoted by $t_{n}(K)$. The sequence $\{t_{n}(K)\}_n$ is not necessarily convergent. However, the sub-additivity property for logarithms of Chebyshev numbers implies the existence of $\lim_{n \rightarrow \infty} t_{n}^{1/n}(K)$. This limit is called the Chebyshev number of $K$, and another classical result of Szeg\"{o} (Corollary 5.5.5 in \cite{Ran95}) says that it is nothing but the logarithmic capacity of $K$, i.e.
\begin{equation}
    \lim_{n\rightarrow\infty}[t_{n}(K)]^{\frac{1}{n}}=
\inf_{n\geq 1}[t_{n}(K)]^{\frac{1}{n}}=Cap(K).
\end{equation}

The alternation theorem allows us to make a connection with the set $S_+(p/q)$ and its qth Chebyshev polynomial, which was discussed in \cite{HJ24}. If $P$ is a real polynomial of degree $n$, then $P$ has an \textit{alternating set} in a compact set $K \subset \mathbb{R}$ if there exists $\{x_k\}_{k=0}^n$ in $K$ satisfying $x_0 < x_1 < \cdots < x_n$ such that
\begin{equation}
    P(x_k) = (-1)^{n-k}||P||_{K}.
\end{equation}
The alternation theorem (Theorem 1.1 in \cite{CSZ17}) says that $T_{n,K}$ has an alternating set in $K$ for a compact $K\subset \mathbb{R}$ and, conversely, any monic polynomial with an alternating set in $K$ is the Chebyshev polynomial on $K$. The following result of Peherstorfer and Totik helps us to understand the logarithmic capacity of the spectrum with a rational frequency and estimate the logarithmic capacity of $S_+(p/q)$. 

\begin{theorem}\label{TotikPeherstorfer}\emph{\textbf{(}\cite{Tot11}, \textit{Theorem 1} \textrm{ and } \cite{Peh11}, \textit{Proposition 1.1}}\textbf{)}
Let $l \in \mathbb{N}$ and $$K=\displaystyle\bigcup_{j=1}^{l}[a_{j},b_{j}].$$ Also, let $T_{n,K}$ and $t_{n}(K)$ denote the nth Chebyshev polynomial of $K$ and the nth Chebyshev number of $K$, respectively, i.e. 
$$
t_{n}(K) := ||T_{n,K}(z)||_{K}.
$$
For a natural number $n\geq 1$ the following are pairwise equivalent.
\begin{itemize}
\item[(a)] $ t_n(K)/[Cap(K)]^n=2$.\\

\item[(b)] $T_{n,K}$ has $n+l$ extreme points on $K$.\\

\item[(c)] $K=\big\{z$ $|$ $T_{n;K}(z)\in[-t_{n}(K),t_{n}(K)]\big\}$.\\

\item[(d)] If $\mu_{K}$ denotes the equilibrium measure of $K$, then each
$\mu_{K}([a_{j},b_{j}])$ for $j=1,2,\dots,l$, is of the form
$q_{j}/n$ with integer $q_{j}$, where $q_{j}+1$ is the number of extreme points on $[a_{j},b_{j}]$.\\

\item[(e)] With $\pi(x)=\prod_{j=1}^{l}(x-a_{j})(x-b_{j})$, the equation
\begin{equation*}
P_{n}^{2}(x)-\pi(x)Q_{n-l}^{2}(x)= c
\end{equation*}
is solvable for the polynomials $P_{n}$ and $Q_{n-l}$ of degree $n$ and $n-l$, respectively, where $c$ is a positive constant.
\end{itemize}
\end{theorem}

The ratio in item (a) is called the Widom factor \cite{GH15}, denoted by $W_{n}(K)$, and the study of its estimates and asymptotics is an active research area \cite{AGH16,CSZ17,CSZ22}.

\subsubsection{Harmonic measure}

The harmonic measure allows us to extend the Dirichlet problem to proper subdomains (nonempty, connected, open subsets) of the extended complex plane $\overline{\mathbb{C}} := \mathbb{C} \cup \{\infty\}$. Letting $D$ be a proper subdomain of $\overline{\mathbb{C}}$ and $\mathcal{B}(\partial D)$ be the $\sigma$-algebra of Borel subsets of $\partial D$, we define a \textit{harmonic measure} for $D$ as a function $$\omega_D:D \times \mathcal{B}(\partial D) \rightarrow [0,1]$$ such that
\begin{itemize}
    \item for each $z \in D$, the map $B \mapsto \omega_D(z,B)$ is a Borel probability measure on $\partial D$, and
    \item if $h: \partial D \rightarrow \mathbb{R}$ is a continuous function, then the representation 
    \begin{equation*}
        h(z) := \int_{\partial D} h(\zeta)~d\omega_{D}(z,\zeta) 
    \end{equation*}
    holds for $z \in D$ such that $h$ is harmonic on $D$ and continuous on $\overline{D}$.
\end{itemize}
The main example is the unit disk, where its harmonic measure is given by the Poisson formula as
\begin{equation*}
    d\omega_{\mathbb{D}}(z,\zeta) = \frac{1}{2\pi}\frac{1 - |z|^2}{|\zeta - z|^2}~ |d\zeta|.
\end{equation*}
A critical observation is that if $\partial D$ is not polar, then $\omega_{D}$ exists and it is unique (Theorem 4.3.2 in \cite{Ran95}). If $\partial D$ is polar, there is no harmonic measure. Motivated by solving the Dirichlet problem, the harmonic measure is defined for continuous functions on the boundary of proper subdomains of the extended complex plane. However, the relation assumed in the definition of the harmonic measure can be extended to bounded Borel functions in $\partial D$ (Theorem 4.3.3 in \cite{Ran95}). 

Harmonic measures satisfy monotonicity in the following sense: Let $D_1$ and $D_2$ be domains in $\overline{\mathbb{C}}$ with non-polar boundaries and $D_1 \subset D_2$. If $B$ is a Borel subset of $\partial D_1 \cap \partial D_2$, then
\begin{equation*}
    \omega_{D_1}(z,B) \leq \omega_{D_2}(z,B)
\end{equation*}
for $z \in D_1$ (Corollary 4.3.9 in \cite{Ran95}). 

Another important result relates harmonic and equilibrium measures as follows: The equilibrium measure of a non-polar compact set $K \subset \mathbb{C}$ is given by 
\begin{equation}
    \mu_{K} = \omega_{D}(\infty,\cdot),
\end{equation}
where $D$ is the connected component of $\overline{\mathbb{C}} \setminus K$ that contains $\infty$ (Theorem 4.3.14 in \cite{Ran95}).

Having a solution of the Dirichlet problem for a domain is characterized by its regularity. If $D$ is a proper subdomain of $\overline{\mathbb{C}}$ and $\zeta \in \partial D$, then a \textit{barrier} at $\zeta$ is a subharmonic function $h$ defined on $D \cap \mathcal{N}$, where $\mathcal{N}$ is an open neighborhood of $\zeta$, satisfying $h < 0$ on $D\cap \mathcal{N}$ and $\lim_{z \rightarrow \zeta} h(z) = 0$. We call a boundary point \textit{regular} if a barrier exists at that point. If every boundary point of $D$ is regular, then $D$ is called \textit{regular}. A domain is regular if and only if the Dirichlet problem is solvable for it (Corollary 4.1.8 in \cite{Ran95}). 

A characterization of regularity of a point is given by equilibrium measures. If we define $K := \overline{\mathbb{C}} \setminus D$ for a proper subdomain $D$ containing $\infty$, then $\zeta \in \partial D$ is regular if and only if $K$ is non-polar and $U^{\mu_K}(\zeta) = I(\mu_K)$ (Theorem 4.2.4 in \cite{Ran95}). Note that $I(\mu_K) = \log(Cap(K))$ if $K$ is compact. 

A geometric characterization given by Wiener uses rings (Theorem 7.1 in \cite{GM05}). If we let $A_{n,\zeta}$ be the ring
\begin{equation*}
    A_{n,\zeta} := \{z~:~2^{-n} \leq |z - \zeta| \leq 2^{-n+1}\}
\end{equation*}
for a boundary point $\zeta \in \partial D$ of the domain $D \subset \overline{\mathbb{C}}$, then $\zeta$ is a regular point if and only if
\begin{equation*}
    \sum_{n=1}^{\infty} \frac{n}{\log\big(1/Cap(A_{n,\zeta} \cap \partial D)\big)} = \infty.
\end{equation*}
For examples, let us focus on compact subsets of the real line. When we say a compact set $K$ is regular, we mean regularity of $\mathbb{C} \setminus K$. A \textit{gap} of a compact $K \subset \mathbb{R}$ is a bounded connected component of $\mathbb{R} \setminus K$. If there are only finitely many gaps and no component of $K$ is a single point, we call $K$ a \textit{finite gap} set. Any finite gap set is regular. More interesting examples appear in the class of homogeneous sets in the sense of Carleson \cite{Car83, SY97}. A compact set $K \subset \mathbb{R}$ is \textit{homogeneous} if there is an $\varepsilon > 0$ such that
\begin{equation*}
    \big|(x-\delta,x+\delta) \cap K\big| \geq \varepsilon \delta
\end{equation*}
for all $\delta < \textrm{diam}(K)$ and all $x \in K$. Homogeneous sets, and in particular positive measure Cantor sets, are regular \cite{CSZ17}. We will use regularity in our proofs in terms of Green's function, which is also related with the harmonic measure and the Dirichlet problem. 

\subsubsection{Green's function}

Green's function is another fundamental concept in the potential theory, providing fundamental solutions of the Laplacian with zero boundary values. In our results, we use Green's function with the pole at infinity defined on the sets whose complements are compact sets, so we modify the definition accordingly. Let $K \subset \mathbb{C}$ be non-polar and compact. The \textit{Green's function} $G_K$ is the unique function defined on $\mathbb{C} \setminus K$ that satisfies the following properties:
\begin{itemize}
    \item $G_K$ is harmonic on $\mathbb{C} \setminus K$,
    \item $G_K(z) = \log|z| + \mathcal{O}(1)$ as $z \rightarrow \infty$,
    \item $G_K(z) \rightarrow 0$ as $z \rightarrow \zeta$ for nearly every $\zeta \in \partial K$.
\end{itemize}

The Green's function $G_K$ is explicitly given by
\begin{equation*}
    G_K(z) = U^{\mu_K}(z) - I(\mu_K) = U^{\mu_K}(z) - \log\left(Cap(K)\right),
\end{equation*}
where $\mu_K$ is the equilibrium measure of $K$. The set $K$ is regular if and only if $G_K$ is zero on $K$ and continuous on $\mathbb{C}$. In the literature, this characterization is also considered as the definition of the \textit{regularity with respect to the Dirichlet problem (or regularity for the potential theory)}.

Let us finish our discussion of potential theory with a maximum principle for subharmonic functions that will be needed in Section \ref{supercritical}.
\begin{theorem}\label{ExtendedMaximumPrinciple}\emph{\textbf{(}\cite{Ran95}, \textit{Theorem 3.6.9}\textbf{)}}
Let $D$ be a domain in $\mathbb{C}$ and let $u$ be a subharmonic function on $D$ which is bounded above. If $\partial D$ is non-polar and $$\limsup_{z \rightarrow \zeta} u(z) \leq 0$$ for nearly every $\zeta \in \partial D$, then $u \leq 0$ on $D$
\end{theorem}

\subsection{Spectral theory}

On the spectral theory side we have two different pictures depending on whether we have a rational or irrational frequency $\alpha$, i.e. a periodic or quasi-periodic operator, respectively. All we review below can be found in textbooks or survey papers, e.g. \cite{Kuc16,Tes00} for periodic and \cite{CFKS87,DF22} for quasi-periodic Schr\"{o}dinger operators. 

For rational $\alpha$, the operator $H_{\alpha,\theta}$ is periodic, so using standard results of Floquet-Bloch theory one gets basic spectral properties of the spectrum, listed below in Proposition \ref{PropertiesofDiscriminant}. These properties are given in terms of the discriminant $t_{p/q}(\theta,E)$ as
\begin{equation}
    \sigma(p/q,\theta) = t_{p/q}(\theta,\cdot)^{-1}\big([-2,2]\big),
\end{equation}
where the discriminant is defined as
\begin{align}
    t_{p/q}(\theta,E) &:= \textrm{tr} \Bigg\{\prod_{j=n-1}^{0} A^E\Big(j\frac{p}{q}+\theta\Big)\Bigg\},\\
        A^E(x) &= \begin{pmatrix}
        E-v(x) & -1\\
        1 & 0 \end{pmatrix} \label{A^Edef}.
\end{align}

\begin{proposition}\label{PropertiesofDiscriminant}\emph{\textbf{(}\cite{JM12}, \textit{Fact 4.1}\textbf{)}}
    Let $p/q$ be a rational number such that $(p,q) = 1$. Then for any $\theta \in \mathbb{T}$ we have the following properties.
    \begin{enumerate}
        \item The discriminant $t_{p/q}(\theta,E)$ is a monic polynomial in $E$ of degree $q$.
        \item The discriminant $t_{p/q}(\theta,E)$ splits over the real line with $q$ distinct roots.
        \item The discriminant $t_{p/q}(\theta,E)$ is greater than or equal to $2$ at all its local maxima and less than or equal to $-2$ at all its local minima.
        \item The spectrum $\sigma(p/q,\theta)$ is the inverse image of $[-2,2]$ under $t_{p/q}(\theta,E)$ on the real line, consists of $q$ possibly touching but not overlapping bands and is purely absolutely continuous.
    \end{enumerate}
\end{proposition}

\subsubsection{Lyapunov exponent}

For irrational $\alpha$, a main tool in our proofs will be the Lyapunov exponent, which characterizes the dynamical properties of solutions to the second-order eigenvalue difference equation
\begin{equation}\label{eigenvalueEquation}
    H_{\alpha,\theta}\psi = E\psi
\end{equation}
over $\mathbb{C}^{\mathbb{Z}}$. For fixed $\alpha \in \mathbb{T}$ and $E \in \mathbb{R}$, we define $A^E : \mathbb{T} \rightarrow SL(2,\mathbb{C})$ as in \eqref{A^Edef}. We call the pair $(\alpha,A^E)$ a Schr\"{o}dinger cocycle and define it as a linear skew-map on $\mathbb{T} \times \mathbb{C}^2$
\begin{equation}
    (\alpha,A^E)(\theta,v) := (\theta + \alpha, A^E(\theta)v), \qquad \theta \in \mathbb{T},~ v \in \mathbb{C}^2.
\end{equation}
Iterations of $(\alpha,A^E)$ produce the solution of \eqref{eigenvalueEquation} with the initial condition $(\psi_0,\psi_{-1})^{\mathrm{T}}$ as
\begin{equation}
\big(\alpha,A^E(\theta)\big)^n\Big(\theta,\begin{pmatrix}
       \psi_0 \\
        \psi_{-1} \end{pmatrix}\Big) = \Big(\theta + n\alpha,\begin{pmatrix}
       \psi_n \\
        \psi_{n-1} \end{pmatrix}\Big), \qquad n \in \mathbb{N}
\end{equation}
The \textit{Lyapunov exponent} of the cocycle $(\alpha,A^E)$ (and hence the discrete Schr\"{o}dinger operator $H_{\alpha,\theta}$) is defined by
\begin{equation}\label{LyapunovExponentDefinition}
    L(\alpha,A^E) := \inf_{n \in \mathbb{N}}\frac{1}{n}\int_{\mathbb{T}}\log\Big|\Big|\prod_{j=n-1}^{0} A^E\Big(j\alpha+\theta\Big)\Big|\Big|d\theta,
\end{equation}
where $||\cdot||$ denotes any matrix norm. By Kingman's sub-additive ergodic theorem, if $\alpha$ is irrational, then we have
\begin{equation}\label{LyapunovExponentIrrational}
    L(\alpha,A^E) = \lim_{n \rightarrow \infty}\frac{1}{n}\log\Big|\Big|\prod_{j=n-1}^{0} A^E\Big(j\alpha+\theta\Big)\Big|\Big|
\end{equation}
for a.e. $\theta \in \mathbb{T}$ and if $\alpha = p/q$ is rational with $(p,q)=1$ we have
\begin{equation}\label{LyapunovExponentRational}
    L(p/q,A^E) = \frac{1}{q}\int_{\mathbb{T}} \log \rho
    \Bigg(\prod_{j=n-1}^{0} A^E\Big(j\frac{p}{q}+\theta\Big)\Bigg)d\theta,
\end{equation}
where $\rho(M)$ is the spectral radius of a matrix $M \in M_2(\mathbb{C})$. 

For fixed irrational $\alpha$, as a function of $E \in \mathbb{C}$, the Lyapunov exponent 
$$L_{\alpha}(E) := L(\alpha,A^E)$$ 
has the following properties:
\begin{enumerate}
    \item $L_{\alpha}(E) \geq 0$ for $E \in \mathbb{C}$
    \item $L_{\alpha}(E) > 0$ for $E \in \mathbb{C} \setminus \Sigma(\alpha)$.
    \item $L_{\alpha}(E)$ is a subharmonic function.
    \item $L_{\alpha}(E)$ is harmonic on $\mathbb{C} \setminus \Sigma(\alpha)$.
    \item $L_{\alpha}(E) = \log|E| + \mathcal{O}(1/|E|)$ as $|E| \rightarrow \infty$
\end{enumerate}

\subsubsection{Density of states}

Another important object associated with $H_{\alpha,\theta}$ is the density of states measure, which will appear in Sections \ref{HarmonicMeasures}, \ref{MultiFrequency} and \ref{Ergodic}. The \textit{density of states measure} of the family of Schr\"{o}dinger operators $\{H_{\alpha,\theta}\}_{\theta \in \mathbb{T}}$ is the measure $dk_{\alpha}$ defined by
\begin{equation}\label{DOSMdefinition}
    \int g ~dk_{\alpha} = \int_{\mathbb{T}} \langle \delta_0, g(H_{\alpha,\theta})\delta_0 \rangle ~d\theta
\end{equation}
for any bounded measurable function $g$. The function $k_{\alpha}$ defined by
\begin{equation}\label{IDSdefinition}
    k_{\alpha}(E) = \int \chi_{(-\infty,E]} ~dk_{\alpha}
\end{equation}
is called the \textit{integrated density of states}. For any irrational frequency $\alpha$, the support of the density of states measure $dk_{\alpha}$ is the spectrum $\Sigma(\alpha)$. 

The Lyapunov expoent and the density of states measure are related through the famous Thouless formula, given by
\begin{equation}\label{ThoulessFormula}
    L(\alpha,A^E) = \int \log |x - E|~dk_{\alpha}(x).
\end{equation}
The Thouless formula connects spectral theory with logarithmic potential theory, since the right-hand side of \eqref{ThoulessFormula} is nothing more than the negative of the logarithmic potential of $dk_{\alpha}$. Moreover, if the Lyapunov exponent is zero on the spectrum (no supercritical energies for analytic potentials), then the density of states measure coincides with the equilibrium measure of the spectrum, i.e. $dk_{\alpha} = \mu_{\Sigma(\alpha)}$. In this case, the logarithmic capacity of the spectrum is provided by a result of Simon, which was given for ergodic families of orthogonal polynomials on the real line \cite{Sim07}. Let us state Simon's result in our setting of quasi-periodic Schr\"{o}dinger operators.

\begin{theorem}\label{SimonCapacity}\emph{\textbf{(}\cite{Sim07}, \textit{Theorem 1.15}\textbf{)}}
     Let $\alpha$ be irrational, $dk_{\alpha}$, $\Sigma(\alpha)$ and $L(\alpha,A^E)$ be the density of states measure, the spectrum, and the Lyapunov exponent of $H_{\alpha,\theta}$, respectively. Also, let $\mu_{\Sigma(\alpha)}$ denote the equilibrium measure of $\Sigma(\alpha)$. Then the following statements are equivalent:
     \begin{itemize}
         \item[(a)] \quad $L(\alpha,A^E) = 0$ for $\mu_{\Sigma(\alpha)}$-a.e. $E$.
         
         \item[(b)] \quad $Cap\big(\Sigma(\alpha)\big) = 1$.
     \end{itemize}
\end{theorem}

\section{Convergence results for rational frequency approximates}\label{Literature}
In this section, we summarize the convergence theorems of the spectrum $\Sigma(\alpha)$ by $S_+(p_n/q_n)$ and the absolutely continuous spectrum $\Sigma_{ac}(\alpha)$ by $S_-(p_n/q_n)$ when the irrational $\alpha$ is approximated by the rational sequence $\{p_n/q_n\}_{n}$. Equivalently, these are the continuity results of $S_+(\alpha)$ and $S_-(\alpha)$ as irrational $\alpha$ is approximated by rational sequences.

We consider convergence results in the Hausdorff metric, the Lebesgue measure, and the logarithmic capacity. The next sections add more results to this list in terms of the logarithmic capacity, the harmonic measure, and the equilibrium measure. Let us start with the Hausdorff metric.

\subsection{Convergence in the Hausdorff metric}

If $K$ is a compact subset of the complex plane and $\Omega_K$ denotes the unbounded component of its complement $\overline{\mathbb{C}}\setminus K$, then the Hausdorff distance between sets $K$ and $F$ is given by
$$
d_H(K,F) = \inf_{\delta > 0} \big\{K \subset F_{\delta},~ F \subset K_{\delta}\big\},
$$
where $K_{\delta}$ denotes the $\delta$-neighborhood of $K$.

Hausdorff metric convergence for $S_+(p_n/q_n)$ and $\Sigma(\alpha)$ dates back to the 1980s and was proved by Avron and Simon \cite{AS83}. They obtained their result for discrete and continuous Schr\"{o}dinger operators with continuous potentials. Let us state their theorem in our case with our notation.

\begin{theorem}\label{AvronSimon}\emph{\textbf{(}\cite{AS83}, \textit{Theorem 3.7}\textbf{)}}
     Let $\alpha \in \mathbb{T}$ be irrational and $p_n/q_n \rightarrow \alpha$ with $(p_n,q_n) = 1$. If $H_{\alpha,\theta}$ is a quasi-periodic Schr\"{o}dinger operator with a continuous potential $v$, then
    \begin{equation}
        \lim_{p_n/q_n \rightarrow \alpha} d_H\Big(S_+(p_n/q_n),\Sigma(\alpha)\Big) = 0.
    \end{equation}
\end{theorem}

For $S_-(p_n/q_n)$ and $\Sigma_{ac}(\alpha)$, in general we do not have convergence in the Hausdorff metric, but partial results in set convergence. 

\subsection{Set convergence and convergence in the Lebesgue measure}

First, let us introduce some notation and definitions. Letting $\mathcal{B}$ denote the quotient space of the Borel subsets of the real line, and letting $\{B_n\}_n$ and $B$ be Borel subset of the real line, we define
\begin{equation}\label{setconvergence}
    \lim_{n \rightarrow \infty} B_n = B \qquad \text{in} \quad \mathcal{B}
\end{equation}
if and only if
\begin{equation}
    \limsup_{n \rightarrow \infty} B_n = \liminf_{n \rightarrow \infty} B_n = B
\end{equation}
if and only if
\begin{equation}
    \lim_{n \rightarrow \infty} \chi_{B_n} = \chi_{B} \qquad \text{Lebesgue a.e.,}
\end{equation}
where
\begin{equation*}
    \liminf_{n \rightarrow \infty} B_n := \bigcup_{n \in \mathbb{N}}\bigcap_{k \geq n} B_k \quad \text{and} \quad \limsup_{n \rightarrow \infty} B_n := \bigcap_{n \in \mathbb{N}}\bigcup_{k \geq n} B_k.
\end{equation*}

In 2011 Shamis obtained the following result as a corollary of continuity of the Lyapunov exponent for quasi-periodic Schr\"{o}dinger operators with analytic potentials, which is due to Bourgain and SJ \cite{BJ02}.

\begin{theorem}\label{Shamis1}\emph{\textbf{(}\cite{Sha11}, \textit{Theorem 1.1}\textbf{)}}
     Let $\alpha \in \mathbb{T}$ be irrational and $p_n/q_n \rightarrow \alpha$ with $(p_n,q_n) = 1$. If $H_{\alpha,\theta}$ is a quasi-periodic Schr\"{o}dinger operator with an analytic potential $v$, then
     \begin{equation}
        \limsup_{p_n/q_n \rightarrow \alpha} S_-(p_n/q_n) \subseteq \Sigma_{ac}(\alpha).
    \end{equation}
\end{theorem}

In the other direction, Shamis proved a weaker result.

\begin{theorem}\label{Shamis2}\emph{\textbf{(}\cite{Sha11}, \textit{Theorem 1.2}\textbf{)}}
     Let $\alpha \in \mathbb{T}$ be irrational and $p_n/q_n \rightarrow \alpha$ with $(p_n,q_n) = 1$. Also, let $H_{\alpha,\theta}$ be a quasi-periodic Schr\"{o}dinger operator with an analytic potential $v$. For $\varepsilon > 0$, let us denote
     \begin{equation}
         S_-(p/q,\varepsilon) = \bigcap_{\theta \in \mathbb{T}} \bigg\{E~\Big|~\text{dist}\Big(E,\sigma(p/q,\theta)\Big) < \varepsilon \bigg\},
     \end{equation}
     where
     \begin{equation}
         \text{dist}(E,K) = \inf_{E' \in K} |E - E'|.
     \end{equation}
     Then
    \begin{equation}
        \Sigma_{ac}(\alpha) \subset \bigcap_{\varepsilon > 0} \liminf_{p_n/q_n \rightarrow \alpha} S_-(p_n/q_n,\varepsilon) 
    \end{equation}
\end{theorem}

In 2012 SJ and Marx obtained convergence for both $S_+$ and $S_-$ according to \eqref{setconvergence} by restricting the approximation to a specific rational sequence converging to a given irrational $\alpha$.

\begin{theorem}\label{JitomirskayaMarx}\emph{\textbf{(}\cite{JM12}, \textit{Theorem 1.1}\textbf{)}}
     Let $\alpha \in \mathbb{T}$ be irrational and $H_{\alpha,\theta}$ be a quasi-periodic Schr\"{o}dinger operator with an analytic potential $v$. Then there is a sequence $p_n/q_n \rightarrow \alpha$ with $(p_n,q_n) = 1$ such that
     \begin{equation}
        \lim_{p_n/q_n \rightarrow \alpha} S_+(p_n/q_n) = \Sigma(\alpha)
    \end{equation}
    and
    \begin{equation}
        \lim_{p_n/q_n \rightarrow \alpha} S_-(p_n/q_n) = \Sigma_{ac}(\alpha).
    \end{equation}
\end{theorem}

The convergence in the Lebesgue measure follows as a corollary of this result.

\begin{theorem}\label{JitomirskayaMarx}\emph{\textbf{(}\cite{JM12}, \textit{Corollary 1.1}\textbf{)}}
     Let $\alpha \in \mathbb{T}$ be irrational and $H_{\alpha,\theta}$ be a quasi-periodic Schr\"{o}dinger operator with an analytic potential $v$. Then there is a sequence $p_n/q_n \rightarrow \alpha$ with $(p_n,q_n) = 1$ such that
    \begin{equation}
        \lim_{p_n/q_n \rightarrow \alpha} \big|S_+(p_n/q_n)\big| = \big|\Sigma(\alpha)\big|
    \end{equation}
    and
    \begin{equation}
        \lim_{p_n/q_n \rightarrow \alpha} \big|S_-(p_n/q_n)\big| = \big|\Sigma_{ac}(\alpha)\big|.
    \end{equation}
\end{theorem}

\subsection{Convergence in logarithmic capacity}

For logarithmic capacity, our first result in \cite{HJ24} showed convergence of $S_+$ in subcritical and critical cases for analytic potentials. This is equivalently the case where the Lyapunov exponent is zero everywhere in the spectrum for analytic quasi-periodic Schr\"{o}dinger operators.

\begin{theorem}\label{Schroedingerresultsubcritical}\emph{\textbf{(}\cite{HJ24}, \textit{Theorem 4.6}\textbf{)}}
 Let $\alpha \in \mathbb{T}$ be irrational and $p_n/q_n \rightarrow \alpha$ with $(p_n,q_n) = 1$. Also, let $H_{\alpha,\theta}$ be a quasi-periodic Schr\"{o}dinger operator with an analytic potential $v$. If $L(\alpha,E) = 0$ for $E \in \Sigma(\alpha)$, then
    \begin{equation}
        \lim_{p_n/q_n \rightarrow \alpha} Cap\big(S_+(p_n/q_n)\big) = Cap\big(\Sigma(\alpha)\big).
    \end{equation}
\end{theorem}

In the special case of the almost Mathieu operator, we also obtained convergence in the supercritical case \cite{HJ24}. 

\begin{theorem}\label{AMOresult1}\emph{\textbf{(}\cite{HJ24}, \textit{Theorem 3.1}\textbf{)}}
    Let $H_{\alpha,\theta}$ be the almost Mathieu operator (AMO) with potential $v(x) = 2\lambda \cos(2\pi x)$, and let $\alpha \in \mathbb{T}$ be irrational such that $p_n/q_n \rightarrow \alpha$ and $(p_n,q_n) = 1$. Then 
    \begin{equation}
        \lim_{p_n/q_n \rightarrow \alpha} Cap\big(S_+(p_n/q_n)\big) = Cap\big(\Sigma(\alpha)\big).
    \end{equation}
\end{theorem}

\begin{theorem}\label{AMOresult2}\emph{\textbf{(}\cite{HJ26-2}\textbf{)}}
    Let $H_{\alpha,\theta}$ be the almost Mathieu operator (AMO) with potential $v(x) = 2\lambda \cos(2\pi x)$ and let $\alpha \in \mathbb{T}$ be irrational such that $p_n/q_n \rightarrow \alpha$ and $(p_n,q_n) = 1$. Then 
    \begin{equation}
        \lim_{p_n/q_n \rightarrow \alpha} Cap\big(S_-(p_n/q_n)\big) = Cap\big(\Sigma_{ac}(\alpha)\big).
    \end{equation}
\end{theorem}

\section{Convergence of capacities for continuous potentials}\label{supercritical}

In our proof of Theorem \ref{Schroedingerresultsubcritical} we used the classical theory of Chebyshev polynomials and Avila's global theory of discrete Schr\"{o}dinger operators with analytic potentials. However, for $H_{\alpha,\theta}$ with continuous potentials we will use a different approach, which is more potential theoretic. The continuity in Hausdorff metric for continuous potentials is a result of Avron and Simon \cite{AS83}. We state their result in our notation.

\begin{theorem}\label{CpotentialContinuity}\emph{\textbf{(}\cite{AS83}, \textit{Theorem 3.7 and Corollary 3.8}\textbf{)}}
    Let $v$ be a continuous function on the unit circle. Let $\sigma(\alpha,\theta)$ be the spectrum of 
    \begin{equation}
    \big(H_{\alpha,\theta}\psi\big)_n = \psi_{n-1} + \psi_{n+1} + v(2\pi\alpha n + \theta)\psi_n,
\end{equation}
and let 
\begin{equation*}
    S_+(\alpha) = \bigcup_{\theta \in \mathbb{T}} \sigma(\alpha,\theta).
\end{equation*}
Then if $\alpha_k \rightarrow \alpha$, we have $E \in S_+(\alpha)$ if and only if there exists $E_k \in S_+(\alpha_k)$ with $E_k \rightarrow E$. Moreover $$\big\{(E,\alpha) ~|~ E \in S_+(\alpha)\big\}$$ is a closed set.
\end{theorem}

First we prove our capacity convergence result in the case of regular $\Sigma(\alpha)$.

\begin{theorem}\label{Schroedingerresultregular}
    Let $\alpha \in \mathbb{T}$ be irrational so that $p_n/q_n \rightarrow \alpha$ and $(p_n,q_n) = 1$. Also, let $H_{\alpha,\theta}$ be a quasi-periodic Schr\"{o}dinger operator with a continuous potential $v$. If $\mathbb{C} \setminus \Sigma(\alpha)$ is regular with respect to the Dirichlet problem and $L(\alpha,E) = 0$ for $E \in \Sigma(\alpha)$ then
    \begin{equation}
        \lim_{p_n/q_n \rightarrow \alpha} Cap\big(S_+(p_n/q_n)\big) = Cap\big(\Sigma(\alpha)\big).
    \end{equation}
\end{theorem}

\begin{proof}
    Since $\Sigma(\alpha)$ and $S_+(p_n/q_n)$ are compact, we can consider their Green's functions $G_{\Sigma(\alpha)}$ and $G_{S_+(p_n/q_n)}$. Now, let's consider following properties of the function $u$ defined as 
    \begin{equation}
        u(z) := G_{\Sigma(\alpha)}(z) - G_{S_+(p_n/q_n)}(z).
    \end{equation}
    \begin{enumerate}
        \item By definition
        \begin{align}
            u(z) &= \log|z| - \log \Big(Cap\big(\Sigma(\alpha)\big)\Big) - \log|z| + \log \Big(Cap\big(S_+(p_n/q_n)\big)\Big) + \mathcal{O} \left(\frac{1}{|z|}\right)\\
            &= \log\Bigg(\frac{Cap\big(S_+(p_n/q_n)\big)}{Cap(\Sigma)}\Bigg) + \mathcal{O} \left(\frac{1}{|z|}\right) \label{Greenfunctiondifasymptotic}
        \end{align}
        as $|z| \rightarrow \infty$. Therefore $u$ is bounded on $\mathbb{C} \setminus S_+(p_n/q_n)$.\\

        \item By definition, $G_{\Sigma(\alpha)}$ is subharmonic everywhere and $G_{S_+(p_n/q_n)}$ is harmonic on $\mathbb{C} \setminus S_+(p_n/q_n)$, so $u$ is subharmonic on $\mathbb{C} \setminus S_+(p_n/q_n)$.\\

        \item By the regularity of $\Sigma(\alpha)$, $G_{\Sigma(\alpha)}(z) \rightarrow 0$ as $z \rightarrow \zeta$ for $\zeta \in \Sigma(\alpha)$ and $G_{\Sigma(\alpha)}$ is continuous everywhere. We also know that any Green's function and hence $G_{S_+(p_n/q_n)}$ is non-negative on $\mathbb{C}$. Therefore 
        \begin{equation*}
            \lim_{z \rightarrow \zeta} u(z) \leq \sup_{z \in S_+(p_n/q_n)} G_{\Sigma(\alpha)}(z)
        \end{equation*}
        for $\zeta \in S_+(p_n/q_n)$.
    \end{enumerate}
    Introducing 
    \begin{equation}
        \Gamma_n := \sup_{z \in S_+(p_n/q_n)} G_{\Sigma(\alpha)}(z)
    \end{equation}
     and using properties (1), (2), (3) with the Extended Maximum Principle (Theorem \ref{ExtendedMaximumPrinciple}) we deduce that $u - \Gamma_n \leq 0$ on $\mathbb{C} \setminus S_+(p_n/q_n)$, i.e.
    \begin{equation*}
        G_{\Sigma(\alpha)}(z) - G_{S_+(p_n/q_n)}(z) \leq \Gamma_n
    \end{equation*}
    for $z \in \mathbb{C} \setminus S_+(p_n/q_n)$. Letting $|z|$ tend to infinity and using \eqref{Greenfunctiondifasymptotic} we get 
    \begin{equation}\label{Gammabound}
        \log\Bigg(\frac{Cap\big(S_+(p_n/q_n)\big)}{Cap(\Sigma)}\Bigg) \leq \Gamma_n.
    \end{equation}
    Noting that $G_{\Sigma(\alpha)}$ is continuous everywhere (by the regularity of $\Sigma(\alpha)$) and $S_+(p_n/q_n)$ converges to $\Sigma(\alpha)$ in Hausdorff metric as $p_n/q_n$ goes to $\alpha$, we conclude that $\Gamma_n \rightarrow 0$ as $p_n/q_n \rightarrow \alpha$. Therefore
    
    \begin{equation}\label{CapacityLowerApproximation}
        \limsup_{p_n/q_n \rightarrow \alpha} Cap\big(S_+(p_n/q_n)\big) \leq Cap(\Sigma)
    \end{equation}

    Next, using Proposition \ref{PropertiesofDiscriminant} we can observe that for any fixed phase $\theta$, the discriminant function $t_{p_n/q_n}(\theta,E)$ is the qth Chebyshev polynomial for the spectrum $\sigma(p_n/q_n,\theta)$ of $H_{p_n/q_n,\theta}$. Also note that after these observation the qth Chebyshev number of $\sigma(p_n/q_n,\theta)$ is 2, so using Theorem \ref{TotikPeherstorfer} we get
    \begin{equation}
        \big[Cap\big(\sigma(p_n/q_n,\theta)\big)\big]^{q_n} = \frac{t_{q_n}\big(\sigma(p_n/q_n,\theta)\big)}{2} = 1,
    \end{equation}
so for any fixed $\theta$, capacity of the spectrum is 1. However, since we assume that the Lyapunov exponent of $H_{\alpha,\theta}$ is 0 on the spectrum $\Sigma(\alpha)$, by Theorem \ref{SimonCapacity}, the capacity of $\Sigma(\alpha)$ is also 1. Therefore, using the monotonicity property of capacity for any fixed $\theta$ we get

    \begin{equation}
        Cap\big(S_+(p_n/q_n)\big) \geq Cap\big(\sigma(p_n/q_n,\theta)\big) = 1 = Cap\big(\Sigma(\alpha)\big) 
    \end{equation}

    and hence

    \begin{equation}\label{CapacityUpperApproximation}
        \liminf_{p_n/q_n \rightarrow \alpha} Cap\big(S_+(p_n/q_n)\big) \geq Cap(\Sigma)
    \end{equation}

Finally, combining \eqref{CapacityLowerApproximation} and \eqref{CapacityUpperApproximation} we obtain

 \begin{equation}
        \lim_{p_n/q_n \rightarrow \alpha} Cap\big(S_+(p_n/q_n)\big) = Cap\big(\Sigma(\alpha)\big).
    \end{equation}

\end{proof}

Next, we consider the case where $\Sigma(\alpha)$ is not regular.

\begin{theorem}\label{Schroedingerresultirregular}
    Let $\alpha \in \mathbb{T}$ be irrational so that $p_n/q_n \rightarrow \alpha$ and $(p_n,q_n) = 1$. If $H_{\alpha,\theta}$ is a quasi-periodic Schr\"{o}dinger operator with a continuous potential $v$ and its Lyapunov exponent $L(\alpha,E)$ is 0 on its spectrum $\Sigma(\alpha)$ then 
    \begin{equation}
        \lim_{p_n/q_n \rightarrow \alpha} Cap\big(S_+(p_n/q_n)\big) = Cap\big(\Sigma(\alpha)\big).
    \end{equation}
\end{theorem}

\begin{proof}
    The main difference compared to the proof of the regular case, Theorem \ref{Schroedingerresultregular}, comes from the fact that $G_{\Sigma(\alpha)}$ may be discontinuous on a polar subset of $\Sigma(\alpha)$. Losing continuity of $G_{\Sigma(\alpha)}$, we use a different approach to obtain
    \begin{equation*}
        \limsup_{\frac{p_n}{q_n} \rightarrow \alpha} Cap\Bigg(S_+\bigg(\frac{p_n}{q_n}\bigg)\Bigg) \leq Cap\big(\Sigma(\alpha)\big).
    \end{equation*}
    Let's define 
    \begin{equation}
        \varepsilon_n := d_H\Bigg(\Sigma(\alpha),S_+\bigg(\frac{p_n}{q_n}\bigg)\Bigg)
    \end{equation}
    and
    \begin{equation}
        \Sigma_{\varepsilon_n}(\alpha) := \displaystyle\bigcup_{E \in \Sigma(\alpha)} \Big\{ z \in \mathbb{C} ~\big|~ |z-E| \leq \varepsilon_n\Big\}.
    \end{equation}
    We know that $\{\varepsilon_n\}_{n}$ is a positive sequence converging to zero, so without loss of generality we assume that $\{\varepsilon_n\}_{n}$ is a non-increasing sequence. Now let's note following properties:
    \begin{enumerate}
        \item $\Sigma_{\varepsilon_n}(\alpha)$ is compact by definition,
        \item $S_+(p_n/q_n) \subseteq \Sigma_{\varepsilon_n}(\alpha)$ by definition of $\varepsilon_n$,
        \item $\Sigma_{\varepsilon_{n+1}}(\alpha) \subseteq \Sigma_{\varepsilon_n}(\alpha)$ by the assumption on $\{\varepsilon_n\}_{n}$ and
        \item $\bigcap_{n} \Sigma_{\varepsilon_n}(\alpha) = \Sigma(\alpha)$ by the fact that $\varepsilon_n \rightarrow 0$.
    \end{enumerate}
    Using items (1), (3) and (4) and the capacity property of nested compact sets \eqref{CapCompactConvergence} we get
    \begin{equation}\label{SigmaEpsilonConvergence}
        Cap\big(\Sigma_{\varepsilon_n}(\alpha)\big) \rightarrow Cap\big(\Sigma(\alpha)\big) \quad \text{as} \quad n \rightarrow \infty.
    \end{equation}
    Next, by item (2) above and the monotonicity property of capacity we have
    \begin{equation}\label{S+SigmaEpsilonInequality}
        Cap\Bigg(S_+\bigg(\frac{p_n}{q_n}\bigg)\Bigg) \leq Cap\big(\Sigma_{\varepsilon_n}(\alpha)\big)
    \end{equation}
    for any $n$. Similarly, by item (3) above and the monotonicity property of the capacity, we have
    \begin{equation}\label{S+SigmaEpsilonNested}
        Cap\big(\Sigma_{\varepsilon_{n+1}}(\alpha)\big) \leq Cap\big(\Sigma_{\varepsilon_n}(\alpha)\big)
    \end{equation}
    for any $n$. 
    Combining \eqref{SigmaEpsilonConvergence}, \eqref{S+SigmaEpsilonInequality} and \eqref{S+SigmaEpsilonNested} we obtain
    \begin{equation}\label{CapacityLowerApproximation2}
        \limsup_{\frac{p_n}{q_n} \rightarrow \alpha} Cap\Bigg(S_+\bigg(\frac{p_n}{q_n}\bigg)\Bigg) \leq Cap\big(\Sigma(\alpha)\big).
    \end{equation}
    For the rest of the proof in the regular case we did not use regularity of $G_{\Sigma(\alpha)}$, so following the same arguments we can get
    \begin{equation}\label{CapacityUpperApproximation2}
        Cap\big(\Sigma(\alpha)\big) \leq \liminf_{\frac{p_n}{q_n} \rightarrow \alpha} Cap\Bigg(S_+\bigg(\frac{p_n}{q_n}\bigg)\Bigg).
    \end{equation}

Again, combining \eqref{CapacityLowerApproximation2} and \eqref{CapacityUpperApproximation2} we obtain

 \begin{equation}
        \lim_{p_n/q_n \rightarrow \alpha} Cap\big(S_+(p_n/q_n)\big) = Cap\big(\Sigma(\alpha)\big).
    \end{equation}

\end{proof}

\section{Convergence for harmonic and equilibrium measures}\label{HarmonicMeasures}

In this section using our theorems from the previous section and a recent result of Totik \cite{Tot24}, we obtain convergence results in harmonic and equilibrium measures.

Let $K$ be a non-polar, i.e. $Cap(K) > 0$, compact set. We define the harmonic measure for $K$ as follows. Let $\Omega_K$ denote the unbounded component of $\overline{\mathbb{C}} \setminus K$. Then the \textit{harmonic measure} of $z \in \Omega_K$ with respect to $\Omega_K$, denoted by $\omega_{\Omega_K}(z,\cdot)$, is the unique measure on $\partial\Omega_K$ with the property that for all functions $h$ that are continuous on the closure of $\Omega_K$ and harmonic in $\Omega_K$
(including the point $\infty$) the representation
\begin{equation*}
h(z) = \int_{\partial\Omega_K} h(\zeta)~d\omega_{\Omega_K}(z,\zeta)   
\end{equation*}
holds. For any $z \in \Omega_K$, this unique measure exists if $K$ is non-polar.

Recently, Totik proved the following convergence result for harmonic measures \cite{Tot24}.

\begin{theorem}\label{Totikharmonicmeasure}\emph{\textbf{(}\cite{Tot24}, \textit{Theorem 1.2}\textbf{)}}
     Let $K$, $K_n$, $n = 1,2,\cdots$ be compact sets, $$Cap(K) > 0,$$ $$d_H(K_n,K) \rightarrow 0$$ and $$Cap(K_n) \rightarrow Cap(K)$$ as $n \rightarrow \infty$. Then for all $z \in \Omega_K$ we have 
     \begin{equation}
         \omega_{\Omega_{K_n}}(z,\cdot) \rightarrow \omega_{\Omega_K}(z,\cdot)
     \end{equation}
     in the weak*-topology as $n \rightarrow \infty$. 
\end{theorem}

Recalling that we have the convergence of the rational frequency approximates in the Hausdorff metric \cite{AS83}, we can use Theorem \ref{Totikharmonicmeasure} to get the following results as a corollary of Theorem \ref{Schroedingerresultirregular}. 

Note that the sets we consider $S_+(p/q)$ and $\Sigma(\alpha)$ are subsets of the real line, so the unbounded components of their complements are the same as their complements. Also note that $Cap(\Sigma(\alpha)) \neq 0$, so we never deal with polar sets. 

\begin{theorem}\label{HarmonicMeasureRegularSchroedingerresult2}
 Let $H_{\alpha,\theta}$ be a quasi-periodic Schr\"{o}dinger operator with a continuous potential $v$, and $\alpha \in \mathbb{T}$ be irrational such that $p_n/q_n \rightarrow \alpha$ and $(p_n,q_n) = 1$. If Lyapunov exponent of $H_{\alpha,\theta}$ is 0 on its spectrum $\Sigma(\alpha)$, then for $z \in \Omega_{\alpha}$
    \begin{equation}
        \omega_{\Omega_{p_n/q_n}}(z,\cdot) \rightarrow \omega_{\Omega_{\alpha}}(z,\cdot)
     \end{equation}
    in the weak*-topology as $p_n/q_n \rightarrow \alpha$, where
     $$\Omega_{p_n/q_n} := \overline{\mathbb{C}} \setminus S_+(p_n/q_n) \qquad \text{and} \qquad \Omega_{\alpha} := \overline{\mathbb{C}}\setminus\Sigma(\alpha).$$ 
\end{theorem}

    For a compact non-polar set $K$, the harmonic measure at $\infty$ is nothing but the equilibrium measure for the set. Therefore, we automatically get

\begin{theorem}\label{EquilibriumMeasureRegularSchroedingerresult2}
Let $H_{\alpha,\theta}$ be a quasi-periodic Schr\"{o}dinger operator with a continuous potential $v$, and $\alpha \in \mathbb{T}$ be irrational such that $p_n/q_n \rightarrow \alpha$ and $(p_n,q_n) = 1$. If Lyapunov exponent of $H_{\alpha,\theta}$ is 0 on its spectrum $\Sigma(\alpha)$, then 
    \begin{equation}
        \mu_{S_+(p_n/q_n)} \rightarrow \mu_{\Sigma(\alpha)}
     \end{equation}
    in the weak*-topology as $p_n/q_n \rightarrow \alpha$, where $\mu_K$ is the equilibrium measure of the compact set $K$.
\end{theorem}

    Note that in our results we assume that the Lyapunov exponent of the quasi-periodic operator $H_{\alpha,\theta}$ is zero on the spectrum. In this case, it is well known that the density of states measure $dk_{\alpha}$ of $H_{\alpha,\theta}$ is the equilibrium measure for the spectrum, so we also obtain

\begin{theorem}\label{DOSEquilibriumMeasureRegularSchroedingerresult2}
Let $H_{\alpha,\theta}$ be a quasi-periodic Schr\"{o}dinger operator with a continuous potential $v$, and $\alpha \in \mathbb{T}$ be irrational such that $p_n/q_n \rightarrow \alpha$ and $(p_n,q_n) = 1$. If Lyapunov exponent of $H_{\alpha,\theta}$ is 0 on its spectrum, then 
    \begin{equation}
        \mu_{S_+(p_n/q_n)} \rightarrow dk_{\alpha}
     \end{equation}
    in the weak*-topology as $p_n/q_n \rightarrow \alpha$, where $\mu_{S_+(p_n/q_n)}$ is the equilibrium measure of $S_+(p_n/q_n)$ and $dk_{\alpha}$ is the density of states measure of $H_{\alpha,\theta}$.
\end{theorem}

\section{Multi-frequency discrete Schr\"{o}dinger operators}\label{MultiFrequency}

In this section, we consider our convergence results for multi-frequency Schr\"{o}dinger operators with continuous potentials, so let us start by defining multi-frequency Schr\"{o}dinger operators. 

Let $\vec{\alpha} = (\alpha_1, \alpha_2, \dots, \alpha_N)$ belong to the $N$-dimensional torus $\mathbb{T}^N := \mathbb{R}^N/\mathbb{Z}^N$. The\textit{ $N$-frequency quasi-periodic Schr\"odinger operator} $$H_{\vec{\alpha},\vec{\theta}}:~ l^2(\mathbb{Z}) \rightarrow l^2(\mathbb{Z})$$ with a continuous potential $v \in \mathcal{C}(\mathbb{T}^N,\mathbb{R})$ is given by
\begin{equation}\label{MultiDimensinoalSchroedinger1}
    \big(H_{\vec{\alpha},\vec{\theta}}\psi\big)_{n} =  \psi_{n+1} + \psi_{n-1} + v(\vec{\theta} + n\vec{\alpha})\psi_{n},
\end{equation}

where $\vec{\theta} \in \mathbb{T}^N$ and $\vec{\alpha} \in \mathbb{T}^N$ are the phase and frequency, respectively.

The spectrum of $H_{\vec{\alpha},\vec{\theta}}$ is denoted by $\sigma(\vec{\alpha},\vec{\theta})$ and is independent of $\vec{\theta}$ for any frequency $\vec{\alpha} = (\alpha_1,\dots,\alpha_N)$ such that $\{1,\alpha_1,\cdots,\alpha_N\}$ is a rationally independent set \cite{AS83}. Therefore we introduce a notation similar to the case of one-frequency for such $\vec{\alpha}$ as 
\begin{equation}
    \Sigma(\vec{\alpha}) := \sigma(\vec{\alpha},\vec{\theta})
\end{equation}
and
\begin{equation}
    S_+(\vec{\alpha}) := \bigcup_{\vec{\theta} \in \mathbb{T}^N}\sigma(\vec{\alpha},\vec{\theta})
\end{equation}
for $\vec{\alpha} \in \mathbb{T}^N$.

Theorem \ref{CpotentialContinuity} of Avron and Simon that we mentioned in Section \ref{supercritical} explicitly applies for multi-frequency discrete Schr\"{o}dinger operators. We first state their result in our notation, so we have the continuity of $S_+$ in the Hausdorff metric.

\begin{theorem}\label{CpotentialContinuityMultiFrequency}\emph{\textbf{(}\cite{AS83}, \textit{Theorem 3.7, Corollary 3.8}\textbf{)}}
    Let $v$ be a real-valued continuous function on the torus $\mathbb{T}^N = \mathbb{R}^N/\mathbb{Z}^N$ and 
    \begin{equation}
    S_+(\vec{\alpha}) = \bigcup_{\vec{\theta} \in \mathbb{T}^N}\sigma(\vec{\alpha},\vec{\theta}).
    \end{equation}
Then if $\vec{\alpha}_k \rightarrow \vec{\alpha}$, we have $E \in S_+(\vec{\alpha})$ if and only if there exists $E_{k} \in S_+(\vec{\alpha}_k)$ with $E_k \rightarrow E$. Moreover $$\big\{(E,\vec{\alpha}) ~|~ E \in S_+(\vec{\alpha})\big\}$$ is a closed set.
\end{theorem}

Note that the proofs of our capacity convergence results in one-frequency setting, Theorem \ref{Schroedingerresultregular} and Theorem \ref{Schroedingerresultirregular}, are potential theoretic and require mainly three assumptions:

\begin{enumerate}
    \item convergence of $S_+(\alpha_k)$ to $\Sigma(\alpha)$ in the Hausdorff metric,
    \item the approximates $S_+(\alpha_k)$ are closed sets and have band-gap structure and
    \item the Lyapunov exponent of $H_{\alpha,\theta}$ is zero on the spectrum $\Sigma(\alpha)$. 
\end{enumerate}

After Theorem \ref{CpotentialContinuityMultiFrequency}, we have the first two properties for the multi-frequency operator $H_{\vec{\alpha},\vec{\theta}}$ if we allow the approximating frequencies to have rational entries. For the third property, we used $Cap(\Sigma(\alpha)) = 1$ through Theorem \ref{SimonCapacity}. However, Theorem \ref{SimonCapacity} is valid for the one-dimensional multi-frequency operators. Therefore, we can follow the proofs of Theorems \ref{Schroedingerresultregular} and \ref{Schroedingerresultirregular}, and get the following result.

\begin{theorem}\label{SchroedingerresultMultiFrequency}
    Let $\vec{\alpha} := (\alpha_1, \cdots, \alpha_N) \in \mathbb{T}^N$ such that $\{1,\alpha_1, \cdots, \alpha_N\}$ is rationally independent and 
    $$\vec{\alpha}_n := \bigg(\frac{p_{1,n}}{q_{1,n}},\frac{p_{2,n}}{q_{2,n}},\dots,\frac{p_{N,n}}{q_{N,n}}\bigg) \longrightarrow \vec{\alpha}$$ as $n \rightarrow \infty$, where $(p_{i,n},q_{i,n}) = 1$ for $1 \leq i \leq N$, $n \in \mathbb{N}$. If $H_{\vec{\alpha},\vec{\theta}}$ is a quasi-periodic Schr\"{o}dinger operator with a continuous potential $v$ and zero Lyapunov exponent on its spectrum $\Sigma(\vec{\alpha})$, then 
    \begin{equation}
        \lim_{n \rightarrow \infty} Cap\big(S_+(\vec{\alpha}_n)\big) = Cap\big(\Sigma(\vec{\alpha})\big).
    \end{equation}
\end{theorem}

The next results are the multi-frequency analogs of Theorem \ref{HarmonicMeasureRegularSchroedingerresult2} and Theorem \ref{DOSEquilibriumMeasureRegularSchroedingerresult2}

\begin{theorem}\label{HarmonicMeasureSchroedingerresultMultiFrequency}
Let $\vec{\alpha} \in \mathbb{T}^N$ such that $\{1,\alpha_1, \cdots, \alpha_N\}$ is rationally independent and 
    $$\vec{\alpha}_n := \bigg(\frac{p_{1,n}}{q_{1,n}},\frac{p_{2,n}}{q_{2,n}},\dots,\frac{p_{N,n}}{q_{N,n}}\bigg) \rightarrow \vec{\alpha}$$ as $n \rightarrow \infty$, where $(p_{i,n},q_{i,n}) = 1$ for $1 \leq i \leq N$, $n \in \mathbb{N}$. If $H_{\vec{\alpha},\vec{\theta}}$ is a quasi-periodic Schr\"{o}dinger operator with a continuous potential $v$ and zero Lyapunov exponent on its spectrum $\Sigma(\vec{\alpha})$, then for all $z \in \Omega_{\vec{\alpha}}$ we have
    \begin{equation}
        \omega_{\Omega_{\vec{\alpha}_n}}(z,\cdot) \rightarrow \omega_{\Omega_{\vec{\alpha}}}(z,\cdot)
     \end{equation}
    in the weak*-topology as $n \rightarrow \infty$, where
     $$\Omega_{\vec{\alpha}_n} := \overline{\mathbb{C}} \setminus S_+\big(\vec{\alpha}_n\big) \qquad \text{and} \qquad \Omega_{\vec{\alpha}} := \overline{\mathbb{C}}\setminus\Sigma(\vec{\alpha}).$$ 
\end{theorem}

\begin{theorem}\label{DOSEquilibriumMeasureSchroedingerresultMultiFrequency}
Let $\vec{\alpha} \in \mathbb{T}^N$ such that $\{1,\alpha_1, \cdots, \alpha_N\}$ is rationally independent and 
    $$\vec{\alpha}_n := \bigg(\frac{p_{1,n}}{q_{1,n}},\frac{p_{2,n}}{q_{2,n}},\dots,\frac{p_{N,n}}{q_{N,n}}\bigg) \rightarrow \vec{\alpha}$$ as $n \rightarrow \infty$, where $(p_{i,n},q_{i,n}) = 1$ for $1 \leq i \leq N$, $n \in \mathbb{N}$. If $H_{\vec{\alpha},\vec{\theta}}$ is a quasi-periodic Schr\"{o}dinger operator with a continuous potential $v$ and zero Lyapunov exponent on its spectrum $\Sigma(\vec{\alpha})$, then
    \begin{equation}
        \mu_{S_+(\vec{\alpha}_n)} \rightarrow dk_{\vec{\alpha}}
     \end{equation}
    in the weak*-topology as $n \rightarrow \infty$, where $\mu_{S_+(\vec{\alpha}_n)}$ is the equilibrium measure of $S_+(\vec{\alpha}_n)$ and $dk_{\vec{\alpha}}$ is the density of states measure of $H_{\vec{\alpha},\vec{\theta}}$.
\end{theorem}

\section{Ergodic Schr\"{o}dinger operators}\label{Ergodic}

In this section, we formulate our convergence theorems for the general setting of ergodic Schr\"{o}dinger operators approximated by families of periodic Schr\"{o}dinger operators. For quasi-periodic operators, we consider our convergence results through the approximation of the irrational frequency by a sequence of rational frequencies. For ergodic operators, we obtain similar convergence theorems in a broader setting by assuming the union spectra of a family of periodic operators converges to the almost sure spectrum of an ergodic operator in the Hausdorff metric. 

Let us recall how we define an ergodic Schr\"{o}dinger operator. If $(\Omega,d\mu)$ is a probability measure space and $T : \Omega \rightarrow \Omega$ is an ergodic transformation, let
\begin{equation}
    (H_{\omega}\psi)_n = \psi_{n+1} + \psi_{n-1} + g(T^n\omega)\psi_n,
\end{equation}
where $g$ is a bounded, measurable function from $\Omega$ to $\mathbb{R}$. Then $\{H_{\omega}\}_{\omega \in \Omega}$ defines an ergodic family of Schr\"{o}dinger operators.

As we discussed in the previous section on the multi-frequency case, we mainly need the following conditions in our results so far:

\begin{enumerate}
    \item convergence of the union spectra of the periodic operators to the spectrum of the quasi-periodic spectrum in the Hausdorff metric,
    \item the approximating union spectra to be a closed set and have band-gap structure and
    \item the Lyapunov exponent of the quasi-periodic operator to be zero on its spectrum. 
\end{enumerate}

In addition to these conditions, we also use Theorem \ref{SimonCapacity}. However, that result was given for ergodic operators. Therefore, assuming the Hausdorff metric convergence of the union spectra and applying the same methods that we discussed in Sections \ref{supercritical} and \ref{HarmonicMeasures}, we get the following theorems.

\begin{theorem}\label{SchroedingerresultErgodic}
Let $\{H_{\omega}\}_{\omega \in \Omega}$ be a family of ergodic Schr\"{o}dinger operators. Also let $\{\{H_{\omega,n}\}_{\omega \in W}\}_{n}$ be a sequence of families of periodic Schr\"{o}dinger operators for $W \subseteq \Omega$ such that
\begin{itemize}
    \item $\Sigma$ denotes the almost sure spectrum of $H_{\omega}$,
    \item $S_+(n) :=  \bigcup_{\omega \in W} \sigma(H_{\omega,n}) $, where $\sigma(H_{\omega,n})$ is the spectrum of $H_{\omega,n}$,
    \item The Lyapunov exponent of $H_{\omega}$ is zero on $\Sigma$, and
    \item $S_+(n)$ converges to $\Sigma$ in Hausdorff metric as $n \rightarrow \infty$.
\end{itemize}
Then
    \begin{equation}
        \lim_{n \rightarrow \infty} Cap\big(S_+(n)\big) = Cap\big(\Sigma\big).
    \end{equation}
\end{theorem}

\begin{theorem}\label{HarmonicMeasureSchroedingerresultErgodic}
Let $\{H_{\omega}\}_{\omega \in \Omega}$ be a family of ergodic Schr\"{o}dinger operators. Also let $\{\{H_{\omega,n}\}_{\omega \in W}\}_{n}$ be a sequence of families of periodic Schr\"{o}dinger operators for $W \subseteq \Omega$ such that
\begin{itemize}
    \item $\Sigma$ denotes the almost sure spectrum of $H_{\omega}$,
    \item $S_+(n) :=  \bigcup_{\omega \in W} \sigma(H_{\omega,n}) $, where $\sigma(H_{\omega,n})$ is the spectrum of $H_{\omega,n}$,
    \item The Lyapunov exponent of $H_{\omega}$ is zero on $\Sigma$, and
    \item $S_+(n)$ converges to $\Sigma$ in Hausdorff metric as $n \rightarrow \infty$.
\end{itemize}
If $\omega_D$ denotes the harmonic measure for the domain $D \subset \overline{\mathbb{C}}$, then for $z \notin \Sigma$
    \begin{equation}
        \omega_{\overline{\mathbb{C}} \setminus S_+(n)}(z,\cdot) \rightarrow \omega_{\overline{\mathbb{C}} \setminus \Sigma}(z,\cdot)
     \end{equation}
     in the weak*-topology as $n \rightarrow \infty$.
\end{theorem}

\begin{theorem}\label{DOSEquilibriumMeasureSchroedingerresultErgodic}
Let $\{H_{\omega}\}_{\omega \in \Omega}$ be a family of ergodic Schr\"{o}dinger operators. Also let $\{\{H_{\omega,n}\}_{\omega \in W}\}_{n}$ be a sequence of families of periodic Schr\"{o}dinger operators for $W \subseteq \Omega$ such that
\begin{itemize}
    \item $\Sigma$ denotes the almost sure spectrum of $H_{\omega}$,
    \item $S_+(n) :=  \bigcup_{\omega \in W} \sigma(H_{\omega,n}) $, where $\sigma(H_{\omega,n})$ is the spectrum of $H_{\omega,n}$,
    \item The Lyapunov exponent of $H_{\omega}$ is zero on $\Sigma$, and
    \item $S_+(n)$ converges to $\Sigma$ in Hausdorff metric as $n \rightarrow \infty$.
\end{itemize}
Then
    \begin{equation}
        \mu_{S_+(n)} \rightarrow dk
     \end{equation}
    in the weak*-topology as $n \rightarrow \infty$, where $\mu_{S_+(n)}$ is the equilibrium measure of $S_+(n)$ and $dk$ is the density of states measure of $H_{\omega}$.
\end{theorem}
An immediate application is a general result for almost periodic potentials. 

Let $V\in \ell^\infty(\mathbb Z,\mathbb R)$ be a uniformly almost periodic potential, and let $\Omega$ be its hull, taken in the $\ell^\infty$-topology, under the shift. For $W\in\Omega$, let
\[
(H_W\psi)_n=\psi_{n+1}+\psi_{n-1}+W(n)\psi_n .
\]
Then by minimality, the spectrum of $H_W$ is independent of $W\in\Omega$; denote this common spectrum by $\Sigma$.

Let $q_j\to\infty$ be a sequence of almost periods such that
\[
\varepsilon_j
:=
\sup_{W\in\Omega}\sup_{n\in\mathbb Z}
|W(n+q_j)-W(n)|
\longrightarrow 0.
\]
For each $W\in\Omega$, let $W^{(j)}$ be the $q_j$-periodic potential obtained by repeating the block
\[
W(0),W(1),\ldots,W(q_j-1),
\]
and define
\[
S_j^{\rm ap}
=
\overline{\bigcup_{W\in\Omega}\sigma(\Delta+W^{(j)})}.
\]

\begin{theorem}
Under the above notation,
\[
d_{\rm H}(S_j^{\rm ap},\Sigma)\longrightarrow 0.
\]

Consequently, if the Lyapunov exponent of the almost periodic family on the hull vanishes on $\Sigma$, then the conclusions of Theorem~\ref{SchroedingerresultErgodic} apply to $S_j^{\rm ap}$. In particular,
\[
\operatorname{Cap}(S_j^{\rm ap})\longrightarrow \operatorname{Cap}(\Sigma),
\]
and the corresponding harmonic and equilibrium measures converge as in Theorems~\ref{HarmonicMeasureSchroedingerresultErgodic} and \ref{DOSEquilibriumMeasureSchroedingerresultErgodic}.
\end{theorem}

\begin{proof}
Since $V$ is uniformly almost periodic, its hull is compact in the $\ell^\infty$-topology and the shift orbit of each $W\in\Omega$ is dense in $\Omega$. Moreover, $H_{TW}$ is unitarily equivalent to $H_W$, and $W\mapsto H_W$ is norm continuous. Hence the spectrum is constant on the hull; we denote it by $\Sigma$.

We prove the two Hausdorff inclusions. First let $E\in\Sigma$ and let $\delta>0$. By the Weyl criterion, there exist $W\in\Omega$ and a finitely supported unit vector $\varphi$ such that
\[
\|(H_W-E)\varphi\|<\delta .
\]
After shifting $W$, we may assume that the support of $\varphi$, enlarged by one site, is contained in $\{0,1,\ldots,q_j-1\}$ for all sufficiently large $j$. On this finite set the potentials $W$ and $W^{(j)}$ agree exactly. Therefore
\[
\|(\Delta+W^{(j)}-E)\varphi\|<\delta ,
\]
and hence $\operatorname{dist}(E,S_j^{\rm ap})<\delta$ for all sufficiently large $j$. This proves
\[
\Sigma\subseteq \liminf_{j\to\infty} S_j^{\rm ap}.
\]

For the reverse inclusion, suppose $E_j\in S_j^{\rm ap}$ and $E_j\to E$. It is enough to show that $\operatorname{dist}(E_j,\Sigma)\to0$. Choose $W_j\in\Omega$ such that $E_j\in\sigma(\Delta+W_j^{(j)})$, up to an error tending to zero. Since $W_j^{(j)}$ is periodic, the Weyl criterion for the periodic operator may be realized by cutting off a Bloch solution over many periods. Thus, for any $M_j\to\infty$, there is a unit vector $\psi_j$ supported on an interval of length of order $M_jq_j$ such that
\[
\|(\Delta+W_j^{(j)}-E_j)\psi_j\|\longrightarrow 0,
\]
provided $M_jq_j\to\infty$.

Choose $M_j\to\infty$ so slowly that
\[
M_j\varepsilon_j\longrightarrow 0.
\]
On an interval containing at most $M_j$ consecutive $q_j$-blocks, the almost-period estimate gives
\[
\|W_j^{(j)}-W_j\|_{\infty}
\leq M_j\varepsilon_j .
\]
Therefore
\[
\|(H_{W_j}-E_j)\psi_j\|
\leq
\|(\Delta+W_j^{(j)}-E_j)\psi_j\|
+
M_j\varepsilon_j
\longrightarrow 0.
\]
Since $\sigma(H_{W_j})=\Sigma$, this implies
\[
\operatorname{dist}(E_j,\Sigma)\longrightarrow0.
\]
Hence every limit point of $S_j^{\rm ap}$ belongs to $\Sigma$, i.e.
\[
\limsup_{j\to\infty}S_j^{\rm ap}\subseteq\Sigma.
\]
Combining the two inclusions gives
\[
d_{\rm H}(S_j^{\rm ap},\Sigma)\to0.
\]
The final statement follows by applying Theorem~\ref{SchroedingerresultErgodic}.
\end{proof}

\section{Counterexamples for approximations of rational frequencies}\label{Counterexamples}

In all the convergence results we consider for quasi-periodic Schr\"{o}dinger operators, the irrational frequency $\alpha$ is approximated by a rational sequence $\{p_n/q_n\}_n$. A natural question is to check for similar results when $\alpha$ is rational. In Hausdorff metric, convergence occurs for rational frequency $\alpha$ too, by the result of Avron and Simon \cite{AS83}. We stated their result in Theorem \ref{AvronSimon} parallel to our setting, but it is valid for any $\alpha$ and any approximating sequence of frequencies $\{\alpha_n\}$. However, this is not true for logarithmic capacity, even for the almost Mathieu operator. To provide our counterexamples, we use explicit formulas for the capacities. For the almost Mathieu operator with the potential $v(x) = 2\lambda \cos(2\pi x)$, when $\alpha = p/q$ (simplified to $(p,q)=1$) is the frequency, we have
\begin{equation}\label{CapS+AMO}
        Cap\big(S_+(p/q)\big) = \big(1 + |\lambda|^{q}\big)^{1/q}.
\end{equation}
On the other hand, for the irrational frequency $\alpha$, capacity of the spectrum is given by
\begin{equation}\label{CapSigmaAMO}
        Cap\big(\Sigma(\alpha)\big) = \begin{cases}
      1 & \text{if }~ |\lambda| < 1 \\
      |\lambda| & \text{if }~ |\lambda| \geq 1
   \end{cases},\\ 
\end{equation}
see \cite{AMS90,Sim07,HJ24}. 
\begin{remark}\label{CounterexampleS+Remark}
    Let $\alpha = p/q \in \mathbb{T}$ be the frequency of the almost Mathieu operator $H_{\alpha,\theta}$ with the potential $v(x) = 2\lambda\cos(2\pi x)$ such that $(p,q) = 1$. Then
    \begin{equation}\label{rationalLimit}
        Cap\big(S_+(p/q)\big) = \big(1 + |\lambda|^{q}\big)^{1/q}.
    \end{equation}
    Also, let $\{p_n/q_n\}$ be a sequence that converges to $\alpha$. Then $\alpha - p_n/q_n$ can be represented as $\widetilde{p}_n/\widetilde{q}_n$ such that $(\widetilde{p}_n,\widetilde{q}_n) = 1$. Therefore
    \begin{equation}\label{rationalApproximants}
        Cap\big(S_+(\widetilde{p}_n/\widetilde{q}_n)\big) = \big(1 + |\lambda|^{\widetilde{q}_n}\big)^{1/\widetilde{q}_n}.
    \end{equation}
    Combining \eqref{rationalLimit} and \eqref{rationalApproximants} we get
    \begin{equation}\label{CounterexampleS+}
         \lim_{p_n/q_n \rightarrow \alpha} Cap\big(S_+(\widetilde{p}_n/\widetilde{q}_n)\big) = \begin{cases}
      1 & \text{if }~ |\lambda| < 1 \\
      |\lambda| & \text{if }~ |\lambda| \geq 1
   \end{cases} \neq \big(1 + |\lambda|^{q}\big)^{1/q} = Cap\big(S_+(p/q)\big) 
    \end{equation}
\end{remark}

The counterexample \eqref{CounterexampleS+} shows that even if we impose the zero Lyapunov exponent condition ($0 < |\lambda| \leq 1$) for the almost Mathieu operator) we do not have capacity convergence.

Another question would be to focus on the absolutely continuous spectrum that is approximated by the intersection spectrum $S_-$. As we discussed in Section \ref{Literature}, there are convergence results in various ways in this setting, in particular some very recent results establish set theoretic and measure-theoretic convergence for the almost Mathieu \cite{GK26} and analytic potentials \cite{JLS26}. 

However, for capacities already Mathieu operators provide similar counterexamples for rational frequencies in the limit. For the capacity of the intersection spectrum we have a similar formula:
\begin{equation}\label{CapS-AMO}
        Cap\big(S_-(p_n/q_n)\big) = \big(1 - |\lambda|^{q_n}\big)^{1/q_n}.
\end{equation}
This formula makes sense only for the subcritical ($ 0 < |\lambda| < 1$) almost Mathieu operator, since the intersection spectrum has finitely many points for the critical ($|\lambda| = 1$) case and it is and empty set for the supercritical ($|\lambda| >1$) case. This is consistent with considering the ac spectrum in the limit, since we know that subcritical, critical and supercritical almost Mathieu operator have purely ac, purely singular and pure point spectrum, respectively. 

\begin{remark}
    Let us consider the operator we had in Remark \ref{CounterexampleS+Remark} with the subcriticality restriction ($0 < |\lambda| <1$). Following the constructions we had in Remark \ref{CounterexampleS+Remark} we get the non-convergence as
    \begin{equation}\label{CounterexampleS-}
         \lim_{p_n/q_n \rightarrow \alpha} Cap\big(S_-(\widetilde{p}_n/\widetilde{q}_n)\big) = 1 \neq \big(1 - |\lambda|^{q}\big)^{1/q} = Cap\big(S_-(p/q)\big) 
    \end{equation}
\end{remark}

\section{Acknowledgments}
This research was supported by NSF DMS-2052899, DMS-2155211, and Simons 896624. BH acknowledges the support of NSF through grant DMS-2052519.

\bibliographystyle{abbrv}
\bibliography{references}

\end{document}